\documentclass[leqno,12pt,a4paper]{article}
\usepackage[a4paper]{geometry}

\usepackage{graphicx, xcolor}
\usepackage{float}

\usepackage[english]{babel}
\usepackage{csquotes}
\usepackage[backend=biber, giveninits=true, url=false, doi=false, eprint=false, isbn=false]{biblatex}

\DeclareFieldFormat[article]{title}{#1}
\DeclareFieldFormat[inproceedings]{title}{#1}
\DeclareFieldFormat[incollection]{title}{#1}
\DeclareFieldFormat[article]{pages}{#1}
\usepackage{amsfonts,amssymb,amsthm,amsmath,amsbsy,bm}
\usepackage{hyperref}
\hypersetup{
    colorlinks=true,
    linkcolor=blue,
    filecolor=magenta,      
    urlcolor=black,
}
\usepackage{bm}
\usepackage{booktabs}
\usepackage{comment}

\newtheorem{theorem}{Theorem}[section]
\newtheorem{lemma}[theorem]{Lemma}

\newtheorem{remark}[theorem]{Remark}

\newcommand{\R}{\mathbb R}
\newcommand{\Rd}{{\mathbb R}^d}

\def\Hmj#1{{{\rm H}^{-1}(#1)}}

\newcommand{\tht}{{\tilde\theta}}
\newcommand{\mtht}{{\widetilde\mtheta}}
\newcommand{\uz}{{\vu^\ast}}

\newcommand{\mA}{{\bf{A}}}

\newcommand{\mI}{{\bf{I}}}
\newcommand{\mE}{{\bf{E}}}
\newcommand{\mM}{{\bf{M}}}
\newcommand{\mR}{{\bf{R}}}
\newcommand{\mQ}{{\bf{Q}}}
\newcommand{\mx}{{\bf{x}}}
\newcommand{\me}{{\bf{e}}}

\newcommand{\msigma}{{\bm{\sigma}}}
\newcommand{\meta}{{\bm{\eta}}}
\newcommand{\mtau}{{\bm{\tau}}}

\newcommand{\mchi}{{\bm{\chi}}}
\newcommand{\mtheta}{{\bm{\theta}}}
\newcommand{\mvartheta}{{\bm{\vartheta}}}
\newcommand{\mthz}{\mtheta^\ast}
\newcommand{\sz}{{\msigma^\ast}}
\newcommand{\szi}{{\msigma_i^\ast}}
\newcommand{\st}{{\widetilde\msigma}}

\newcommand{\vu}{{\sf u}}
\newcommand{\vv}{{\sf v}}

\def\Mat#1#2{{{\rm M}_{#1}(#2)}}

\newcommand{\lA}{{\cal A}}
\newcommand{\lK}{{\cal K}}
\newcommand{\lB}{{\cal B}}

\newcommand{\lT}{{\cal T}}
\newcommand{\lTt}{\widetilde{\cal T}}
\newcommand{\lTz}{{\cal T}^\ast}

\newcommand{\lS}{{\cal S}}

\newcommand{\lSz}{{\cal S}^\ast}

\newcommand{\vf}{{\sf f}}
\newcommand{\mnul}{{\bf{0}}}
\def\oi#1#2{\langle#1,#2\rangle}        
\def\zi#1#2{[#1,#2]}               
\def\Lb#1{{{\rm L}^\infty(#1)}}

\def\pL#1#2{{{\rm L}^{#1}(#2)}}

\def\Ld#1{{{\rm L}^{2}(#1)}}
\newcommand{\dv}{{\sf div\thinspace}}
\def\Hmj#1{{{\rm H}^{-1}(#1)}}
\newcommand{\lamm}{\lambda_-(\mtheta)}
\newcommand{\lammy}{\lambda_-(\mtheta(y))}

\newcommand{\lamp}{\lambda_+(\mtheta)}

\newcommand{\lammzr}{\lambda_-(\mtheta^\ast(\rho))}
\newcommand{\lammz}{\lambda_-(\mtheta^\ast)}

\newcommand{\Kth}{\lK(\mtheta)}
\newcommand{\Kthx}{\lK(\mtheta(\mx))}

\def\thez#1{\theta^\ast_{#1}}
\def\pLd#1{{{\rm L}^{2}(#1)}}
\def\Hjnl#1{{{\rm H}^{1}_0(#1)}}

\def\MR#1{{\rm M}_#1(\R)}

\newcommand{\Sym}{{\rm Sym_d}}

\newcommand{\K}{{K}}
\renewcommand{\div}{\mathsf{div}}

\def\Svaki#1{\left(\forall\,#1\right)}

\def\dscon{\relbar\joinrel\rightharpoonup}

\def\povrhksk#1{\smash{
        \mathop{\dscon}\limits^{#1}}}

\newcommand{\supp}{\mathrm{ess \;supp\,}}

\title{Classical Optimal Designs for Stationary Diffusion with Multiple Phases}
\author{Matko Grbac\footnote{Department of Mathematics, Faculty of Science,  University of Zagreb, Croatia, 
\texttt{magrbac@math.hr}} \and 
Ivan Ivec\footnote{Faculty of Metallurgy,  University of Zagreb, Croatia, 
\texttt{iivec@simet.hr}} \and 
Marko Vrdoljak\footnote{Department of Mathematics, Faculty of Science,  University of Zagreb, Croatia, 
\texttt{marko@math.hr}}}

\begin{document}

\maketitle

\begin{abstract}
We study optimal design problems for stationary diffusion involving one or more state equations and mixtures of an arbitrary number of anisotropic materials. Since such problems typically do not admit classical solutions, we adopt a homogenization-based relaxation framework. 

The objective considered is the maximization of a weighted sum of the energies associated with each state equation, with particular emphasis on identifying cases in which the optimal design is classical, that is, of bang-bang type, composed solely of the original pure materials. 
Such cases provide valuable benchmarks for numerical methods in optimal design.

A simplified optimization problem expressed in terms of local material proportions is analyzed through a dual formulation in terms of fluxes. Using a saddle-point characterization, we establish a complete description of its optimal solutions.
The proposed approach is applied in detail to spherically symmetric problems. In the case of a ball, the method yields explicit classical solutions of the homogenization-based relaxation problem. 
\end{abstract}


\section{Introduction}

Optimal design problems aim to determine the spatial arrangement of two or more given materials so that the 
resulting composite exhibits the best possible performance, usually expressed through the minimization of 
some integral functional, depending also on the solution of the corresponding partial differential equation. 
Since such problems often lack classical solutions, one is naturally 
led to consider relaxed formulations. The homogenization method provides one such relaxation: generalized 
(composite) materials are introduced as limits of fine-scale mixtures of the original phases.
This approach originates from the pioneering work of Spagnolo \cite{S68} and Murat-Tartar \cite{MT1} on 
$G$- and $H$-convergence for stationary diffusion equations.

In this paper, we consider optimal design problems for stationary diffusion involving one or more state 
equations and mixtures of an arbitrary number of anisotro\-pic materials. The homogenization approach for a single state equation and two isotropic materials was developed 
in \cite{MT}, and extended to an arbitrary number of anisotro\-pic materials in \cite{Ta1}. Problems involving multiple state equations with two isotropic phases were examined in \cite{A}. The main difficulty in the case of anisotropic materials is the lack of a complete characterization of the  G-closure, which is fully understood in the case of two isotropic materials \cite{Testfin,LC86}

For single-state problems, relaxed solutions correspond to simple laminates \cite{R78,MT}, even in the presence 
of multiple anisotropic phases \cite{Ta1}. This makes it possible to replace the full homogenization-based 
relaxation with a simpler one involving convex combinations of the admissible designs, a formulation connected 
with sliding controls in optimal control theory \cite{CM70,G62}. However, such simplifications generally fail for 
multiple-state problems: relaxed optimal designs may require higher-order sequential laminates \cite{A,AVmpe,AVfer,Vnarwa}.

The objective of the present work is 
to maximize a weighted sum of the energies associated with each state equation 
and to identify situations in which the optimal design is classical (or bang-bang solution), that is composed solely of the original 
pure materials. Such examples are 
valuable as benchmark tests for numerical algorithms in optimal design, especially shape-sensitivity methods \cite{SZ,DZ}.

Although classical solutions rarely occur, they are standard  in the case of energy maximization with two isotropic phases in spherically symmetric problems. The first such example involved a single state equation on a ball with a constant right-hand side, where the 
optimal design places the better conductor in a smaller concentric ball \cite{MT,Ta4}. This model also 
describes the maximization of the torsional rigidity of a cylindrical rod with uniform cross-sections composed of two isotropic elastic materials. By contrast, the analogous problem on 
a square admits no classical solution \cite{Glow,GKR,LCtop}. Moreover, the ball is the only simply connected C$^{1,1}$ domain
for which the classical solution appears, in case of constant right-hand side \cite{CDmax}, see also \cite{MT}. 
In the same setting, when the objective is energy minimization, 
the optimal design is not classical \cite{Ta4,CDmin,BurVrd2019}. 

Spherically symmetric problems on a ball with multiple states and two isotropic materials admit classical solutions, which can be computed explicitly in a straightforward manner \cite{Vsiam}. Additional examples of problems possessing classical solutions can be found in \cite{ABFL, CLM12}.

The remainder of the paper is organized as follows. In Section 2, we introduce the proper relaxation \eqref{odp} of the optimal design problem and consider two additional problems aimed at energy maximization: one posed on a larger admissible set \eqref{odpb}, and a much simpler problem \eqref{eq:opt_des_T}, formulated purely in terms of the local material proportions. While problems \eqref{odp} and \eqref{odpb} are equivalent for single-state problems, even for an arbitrary number of anisotropic materials \cite{Ta1,CDbook}, this equivalence generally fails in the case of multiple state equations.

The main focus of the paper is to solve  problem \eqref{eq:opt_des_T}, by exploiting its dual interpretation in terms of fluxes \cite{MT,GKR}. Section 3 describes this dual formulation and presents an application of the saddle-point theorem to characterize its solutions (Remark \ref{rem:solT}), as well as connections (Remark \ref{rem:J2}) with the solutions of problems \eqref{odpb} and \eqref{odp}. Remark \ref{rem:solT} introduces an auxiliary optimization problem, whose solution is addressed in Section 4. Although the primary interest of the paper lies in its bang–bang solutions, whose existence follows from the results of Artstein \cite{Art80}, we are able to characterize the full set of solutions.

Section 5 focuses on spherically symmetric problems, where the proposed approach can be fully applied, particularly when the domain is a ball, leading to solutions of the relaxed problem \eqref{odp}. Finally, Section 6 presents several examples illustrating the method.

\section{Formulation and Relaxation of the Problem}

Let $\Omega\subseteq\R^d$ be an open, bounded domain occupied by  $N$ distinct materials.
We assume that each material fills a measurable region of $\Omega$.
Accordingly, we introduce the measurable function $\mchi=(\chi_1,\ldots,\chi_N)$,  where $\chi_i$ denotes the characteristic function
of the  $i$-th material phase. The collection of all such admissible functions constitutes the set $\Lb{\Omega;\{\me_1\ldots,\me_N\}}$ where
$\me_i$ denotes the $i$-th canonical basis vector of $\R^N$.

The quantity of $i$-th material is prescribed exactly to $q_i$, so we have 
\begin{equation}\label{eq:kol}
\sum_{i=1}^N q_i=\mu(\Omega),
\end{equation}
where $\mu$ denotes the $d$-dimensional Lebesgue measure.
For $i=1,\ldots,N$ let $\mM_i$ be a symmetric positive definite matrix representing the conductivity tensor of material $i$. In the simplest case of isotropic materials, the conductivity reduces to a scalar multiple of the identity matrix.
In the general, anisotropic case, local rotations can be applied to construct the overall conductivity field
\begin{equation}\label{cond}
\mA(\mx)=\sum_{i=1}^N\chi_i(\mx)\mR^\tau(\mx)\mM_i\mR(\mx),\;\mx\in\Omega,
\end{equation}
where $\mR\in\Lb{\Omega; {\rm SO}(\Rd)}$ denotes a measurable rotation field.

The body is subjected to distinct heat source terms
$f_1,\ldots,f_m\in\Hmj\Omega$. By the  Lax-Milgram's lemma,  boundary value problems
\begin{equation}\label{sem}
\left\{
\begin{array}{l}
-\dv(\mA\nabla u_i)=f_i\\
u_i\in {\rm H}^1_0(\Omega)\,\\
\end{array}\qquad\qquad i=1,\ldots, m
\right.
\end{equation}
have unique solutions (temperatures)
$u_1,\ldots,u_m\in\Hjnl\Omega$.
We shall use notation $\vf=(f_1,\ldots, f_m)$ and $\vu=(u_1,\ldots, u_m)$. The objective is to determine an arrangement of given materials that maximizes the weighted sum of the corresponding energy functionals 
\begin{equation}\label{cf}
\displaystyle\sum_{i=1}^m\mu_i\int_\Omega \mA\nabla u_i\cdot\nabla u_i\,d\mx=\sum_{i=1}^m\mu_i\int_\Omega f_iu_i\,d\mx,
\end{equation}
where the latter expression is interpreted as the dual pairing between  $\Hmj\Omega$
and $\Hjnl\Omega$ whenever the integral is not classically defined.
The weights $\mu_i>0$ are prescribed.

Even  in the simplest isotropic setting with $N=2$ materials  and a single load case ($m=1$), it is well known that the problem may fail to admit a solution \cite{MT,GKR}. 
Consequently, an appropriate relaxation must be introduced, typically formulated within the framework  of homogenization theory \cite{MT,Ta2,A,CDbook}. 

Within this framework, we introduce composite  materials. For a local volume fraction $\mvartheta\in\K$, with
\[
\K:=\left\{\mtheta\in\zi01^N:\sum_{i=1}^N\theta_i=1\right\},
\]
 and $\mA\in\Mat n\R$ we say that $\mA$ belongs 
to the set $\lK(\mvartheta)$ if and only if  there exist a sequence $(\mchi^n)\subseteq\Lb{\Omega;\{\me_1\ldots,\me_N\}}$ and a sequence of rotation fields
$(\mR_n)\subseteq\Lb{\Omega;{\rm SO}(\Rd)}$ such that sequence 
\begin{equation}\label{condn}
\mA^n=\sum_{i=1}^N\chi^n_i\mR_n^\tau\mM_i\mR_n 
\end{equation}
$H$-converges 
to a limit matrix field $\mA^\infty$, with $\mchi^n\povrhksk\ast\,\mtheta$ and such that $\mA=\mA^\infty(\widetilde\mx)$,  
$\mvartheta=\mtheta(\widetilde\mx)$, for some Lebesgue point $\widetilde\mx$ of both $\mA^\infty$ and $\mtheta$. Since each $\mA^n$ is a symmetric matrix almost everywhere on $\Omega$, the same hold for its $H$-limit $\mA^\infty$.

As a consequence (see \cite[Lemma 39--40]{Ta2}), the limit matrix field satisfies $\mA^\infty(\mx)\in\Kthx$ for almost every $\mx\in\Omega$, and, conversely, any measurable pair $(\mtheta,\mA^\infty)$ with this property can be approximated by a sequence of such mixtures, in the sense that  $\mA^n\povrhksk H\,\mA^\infty$ and $\mchi^n\povrhksk\ast\,\mtheta$.  

The principal difficulty arises from the fact that the set  $\lK(\mvartheta)$ is not known in explicit form, except in the special case of mixtures of two isotropic materials. However, in the single-state setting this difficulty can be effectively circumvented by employing classical Voigt-Reuss bounds on the effective conductivities of mixtures in  $\lK(\mvartheta)$, which we recall below.

If $\lambda_1(\mM)$ and $\lambda_d(\mM)$  denote, respectively, the smallest and the largest eigenvalue of a symmetric matrix  $\mM$,
we introduce the quantities
\begin{align*}
&\frac1{\lambda_{-}(\mtheta)} = \sum_{i = 1}^{N} \frac{\theta_{i}}{\lambda_{1}(\mM_{i})}\\
&\lambda_{+}(\mtheta) = \sum_{i = 1}^{N} \theta_{i}\lambda_{d}(\mM_{i}),
\end{align*}
so that $0<\lamm\leq\lamp$ for every $\mtheta\in\K$.
For any  $\mtheta\in\Lb{\Omega;\K}$ and any $\mA\in\Kth$ almost everywhere on $\Omega$,
we obtain the uniform (ellipticity) bounds
\begin{equation}\label{ocj}
\lamm\mI\leq\mA\leq\lamp\mI\,\;\hbox{ almost everywhere on }\Omega.
\end{equation}

We are now in a position to formulate the  relaxed problem. Let
\begin{align*}
\lA=\Bigl\{(\mtheta,\mA)\in\Lb{\Omega;\K\times\Sym}:&\int_\Omega\theta_i\,d\mx = q_i, i=1,\ldots,N, \\
&\mA(\mx)\in\lK(\mtheta(\mx))\,,\,\hbox{ a.e. }{\mx\in\Omega}\Bigr\}\,,
\end{align*}
and consider the following  relaxed optimal design problem
\begin{equation}\label{odp}\tag{A}
\begin{array}{ll}
J(\mtheta,\mA)&=\displaystyle\sum_{i=1}^m\mu_i\int_\Omega f_iu_i\,d\mx\to \max\,\\
&\vu\hbox{ solves }(\ref{sem}) \hbox{ with }(\mtheta,\mA)\in\lA.
\end{array}
\end{equation}
 
From the bounds \eqref{ocj} it follows that $\lA\subseteq\lB$, where
\begin{align*}
\lB=\Bigl\{(\mtheta,\mA)\in\Lb{\Omega;\K\times\Sym}:&\int_\Omega\theta_i\,d\mx = q_i, i=1,\ldots,N, \\
&\lamm\mI\leq\mA\leq\lamp\mI\,,\,\hbox{ a.e. on }{\Omega}\Bigr\},
\end{align*}
and we introduce the same optimization problem, but on this larger domain:
\begin{equation}\label{odpb}\tag{B}
\begin{array}{ll}
J(\mtheta,\mA)&=\displaystyle\sum_{i=1}^m\mu_i\int_\Omega f_iu_i\,d\mx\to \max\,\\
&\vu\hbox{ solves }(\ref{sem}) \hbox{ with }(\mtheta,\mA)\in\lB.
\end{array}
\end{equation}

Owing to the property
\begin{equation}\label{kebe}
\Svaki{\mtheta\in\K}\Svaki{\mE\in\Rd}\quad\lK(\mtheta)\mE=\lB(\mtheta)\mE,
\end{equation}
it can be shown \cite{Ta1,CDbook} that, for a single state equation ($m=1$), the  maximal values of $J$ over $\lA$ and its superset $\lB$ coincide, even for general cost functionals $J$. More precisely, the equality in \eqref{kebe}, and consequently the maximal value of $J$ on $\lA$, is attained by a simple laminate, whose effective conductivity has  eigenvalue  $\lamm$ in direction $E$, and $\lamp$ in at least one orthogonal direction. In the special case of the maximizing the energy functional this result yields the conclusion that the same maximal value is achieved by designs $(\mtheta,\mA)$ with $\mA=\lamm\mI$ almost everywhere in $\Omega$.
In other words, for given $\mtheta$ instead of state equation \eqref{sem} we study 
\begin{equation}\label{semt}
\left\{
\begin{array}{l}
-\dv(\lamm\nabla u_i)=f_i\\
u_i\in {\rm H}^1_0(\Omega)\,\\
\end{array}\qquad\qquad i=1,\ldots, m.
\right.
\end{equation}

In analogy with our setting, if we define $I(\mtheta):=J(\mtheta,\lamm\mI)$, the third natural problem is the maximization of $I$ over the set
\[
    \lT = \left\{ \mtheta \in \Lb{\Omega; \K)} :\int_\Omega\theta_i\,d\mx = q_i, i=1,\ldots,N \right\}.
\]
The third problem is therefore stated as
\begin{equation}\label{eq:opt_des_T}\tag{T}
    \begin{split}
        I(\mtheta) &: = \sum_{i=1}^m \mu_i \int_\Omega f_i u_i dx \to\max\\
        &\vu\hbox{ solves }(\ref{semt}) \hbox{ with }\mtheta\in\lT.
    \end{split}
\end{equation}
\section{Interpretation through a minimax problem}

In \cite{Vsiam} the same optimization problems were analyzed in the simpler setting of mixtures of two isotropic materials. In that case, 
a scalar function $\theta$ represents the volume fraction of one material, whereas in the present work the local mixture is encoded by a vector-valued function $\mtheta$, whose components sum to 1. 
Aside from this increased dimensionality, the structure of the relaxed problems considered here is formally analogous to the two-phase case once the expressions for the effective lower and upper bounds $\lamm$ and $\lamp$ have been appropriately generalized. Under this correspondence, the maximization problems over the admissible sets $\lA$, $\lB$ and $\lT$ take the same abstract form as those studied in \cite{Vsiam}. Consequently, our strategy mirrors the approach developed there, while omitting technical details that overlap entirely with the known two-phase analysis.

We employ the dual of variational formulation, expressed in terms of the heat fluxes $\msigma_i$, in order to rewrite the maximization of $I$ over $\lT$ as a minimax problem. For each admissible  $\mtheta$, we have
\begin{align}
I(\mtheta)&=-\left(-\sum_{i=1}^m\mu_i\int_\Omega f_iu_i\,d\mx\right)\nonumber\\ 
&=-\sum_{i=1}^m\mu_i\int_\Omega \lamm|\nabla u_i|^2 - 2f_iu_i\,d\mx\nonumber\\
&=-\min_{\vv\in\Hjnl{\Omega;\R^m}}\sum_{i=1}^m\mu_i\int_\Omega \lamm|\nabla v_i|^2 - 2f_iv_i\,d\mx\label{varfor}\\
&=-\max_{\msigma\in\lS}\left(-\sum_{i=1}^m\mu_i\int_\Omega \frac{|\msigma_i|^2}{\lamm}\,d\mx\right)\label{dual1}\\
&=\min_{\msigma\in\lS}\sum_{i=1}^m\mu_i\int_\Omega \frac{|\msigma_i|^2}{\lamm}\,d\mx.\nonumber
\end{align}
Equality \eqref{varfor} follows from the  standard variational formulation of the boundary value problem \eqref{semt}, and its solution $\vu$ is the unique minimizer in \eqref{varfor}. Equality \eqref{dual1}
is a consequence of the inequality (for any uniformly positive and bounded measurable matrix field $\mA$)
\begin{equation}\label{fenchel}
-\int_\Omega{\mA}^{-1}\sigma_i\cdot\sigma_i\,d\mx\leq\int_\Omega\mA\nabla v_i\cdot\nabla v_i-2f_iv_i\,d\mx,
\end{equation}
valid for any $v_i\in\Hjnl\Omega$ and $\sigma_i\in\Ld{\Omega;\Rd}$ satisfying $-\div\sigma_i=f_i$, with equality if and only if $\sigma_i=\mA\nabla v_i$. 
Therefore, the set $\lS$ appearing in equality \eqref{dual1} is defined  by
\[
\lS=\{\msigma\in\pLd{\Omega;\Rd}^m:-\dv\msigma_i=f_i, i=1,\ldots, m\}.
\]
Although inequality \eqref{fenchel} can be proved directly, it can also be interpreted as an instance of the Fenchel–Young inequality associated with the quadratic functional $\int_\Omega\mA\nabla v_i\cdot\nabla v_i\,d\mx$ (see \cite[Example IV.2.1]{ET}).

To summarize these conclusions, we obtain the following lemma.

\begin{lemma}\label{lem:dual}
For any $\mtheta\in\lT$ we have
\[
I(\mtheta)=\min_{\msigma\in\lS}\sum_{i=1}^m\mu_i\int_\Omega \frac{|\msigma_i|^2}{\lamm}\,d\mx
\]
and the minimum is attained at unique $\msigma\in \lS$ given by $\msigma_i=\lamm\nabla u_i$, where $u_i$ solves \eqref{semt}, for any $i=1,\ldots, m$.
\end{lemma}

\begin{remark}\label{rem:J1}
The same reasoning yields an analogous result for the energy functional $J$ on $\lB$:
for any $(\mtheta,\mA)\in\lB$,
\[
J(\mtheta,\mA)=\min_{\msigma\in\lS}\sum_{i=1}^m\mu_i\int_\Omega \mA^{-1}\msigma_i\cdot\msigma_i\,d\mx
\]
and the minimum is attained at a unique $\msigma\in \lS$ given by $\msigma_i=\mA\nabla u_i$, where $u_i$ denotes the state function associated with $\mA$ through  \eqref{sem}.
\end{remark}
Finally, we express the maximization of $I$ over $\lT$, as the following max-min problem:
\begin{align}\label{minimax}
\max_{\mtheta\in\lT} I(\mtheta)&=\max_{\mtheta\in\lT} \min_{\msigma\in\lS}\sum_{i=1}^m\mu_i\int_\Omega \frac{|\msigma_i|^2}{\lamm}\,d\mx\nonumber \\
&=\max_{\mtheta\in\lT} \min_{\msigma\in\lS}\sum_{i=1}^m\mu_i\int_\Omega \sum_{j=1}^N\frac{\theta_j}{\lambda_1(\mM_j)}|\msigma_i|^2\,d\mx.
\end{align}

For this max-min problem, one can notice that the functional
\[
H(\mtheta,\msigma):=\sum_{i=1}^m\mu_i\int_\Omega  \frac{|\msigma_i|^2}{\lamm} \,d\mx
\]
is concave (in fact, linear) with respect to $\mtheta$ and uniformly convex and quadratic in $\msigma$. With respect to L$^\infty$ weak-* topology, the set $\lT$ is compact, and $H$ is continuous.
Consequently,  $H$ admits a saddle point \cite[Chapter 6]{ET} (see also \cite{F53})
and 
the set of all its saddle points is of the form $\lTz\times\lSz$, with convex sets $\lTz\subseteq\lT$ and 
$\lSz\subseteq\lS$. 
Moreover, the strict convexity of $H$ in $\msigma$ implies that $\lSz$ is a singleton: $\lSz=\{\sz\}$, so  saddle-points $(\mthz,\sz)$ are characterized by: 
\begin{equation}\label{saddle}
H(\mtheta,\sz)\leq H(\mthz,\sz)\leq H(\mthz,\msigma), \quad\mtheta\in\lT, \msigma\in\lS.
\end{equation}
We denote the (common) value $H(\mthz,\sz)$  by $H^*$.

\begin{lemma}\label{lem:maxB}
The set of optimal solutions of problem \eqref{eq:opt_des_T} coincides with $\lTz$. 

\end{lemma}

\begin{proof}
Let us denote by $\lTt$ the set of all maximizers of $I$ over $\lT$. The inclusion  $\lTt\subseteq\lTz$ follows from  
\eqref{minimax}: if  $I(\mtht)=\max_\lT I$, then $(\mtht,\sz)$ is a saddle point of $H$ by \cite[Proposition 1.2]{ET}. 

For the proof of converse inclusion $\lTz\subseteq\lTt$, let $\mthz\in\lTz$ and let $\mtheta\in\lT$ be arbitrary. 
Denote by $\vu$ the solution of \eqref{semt},
and define $\msigma\in\lS$  by $\msigma_i=\lamm\nabla u_i$, for $i=1,\ldots,m$.
Then we obtain
\[
I(\mtheta)=H(\mtheta,\msigma)\leq H(\mtheta,\sz)\leq H(\mthz,\sz) = \min_{\mtau\in\lS}H(\mthz,\mtau)=I(\mthz),
\]
which proves the claim. Here, the first and last equalities, as well as the first inequality, follow
from Lemma \ref{lem:dual}. The second inequality and the second equality 
are consequences of the saddle-point property \eqref{saddle} of $H$.

\end{proof}

\begin{remark}\label{rem:solT}
By the previous lemma, a function $\mthz$ solves 
\eqref{eq:opt_des_T} if and only if it satisfies both inequalities in \eqref{saddle}. The first inequality
is equivalent to the statement that $\mthz$ is a solution of the optimization problem
\begin{equation}\label{max_1}
\begin{aligned}
  H(\mtheta,\sz)&=\sum_{i=1}^m\mu_i\int_\Omega \frac{|\szi|^2}{\lamm}\,d\mx\to\max\\
  \mtheta&\in\lT,
\end{aligned}
\end{equation}
where the flux field $\sz$ is the unique minimizer provided by the  saddle point theorem.
The next Section is therefore devoted to the solution of problem \eqref{max_1} under the assumption that $|\sz|$ is known. 

The second inequality in \eqref{saddle} is, by Lemma \ref{lem:dual}, equivalent to 
\begin{equation}\label{eq:optT}
\szi=\lammz\nabla u^*_i,\;i=1,\ldots,m,
\end{equation}
where $\uz$ denotes the solution of \eqref{semt} corresponding to $\mthz$.
\end{remark}

We note that the  problem of maximizing $J$ over $\lA$ cannot be treated by the same approach, since the admissible set $\lA$ is not convex \cite[Example 2.1]{Vsiam}. However, the same problem on a bigger admissible set $\lB$ admits similar treatment via saddle-point theorem \cite{Vsiam}. The following simple argument will have important consequences later.

\begin{remark}\label{rem:J2}
Let $(\mtheta,\mA)\in\lB$, and let $u_i$ be the solution of \eqref{sem}, with associated fluxes $\sigma_i=\mA\nabla u_i$,  $i=1,\ldots m$. By Remark \ref{rem:J1}  we have
\[
J(\mtheta,\mA)=\sum_{i=1}^m\mu_i\int_\Omega \mA^{-1}\msigma_i\cdot\msigma_i\,d\mx\leq \sum_{i=1}^m\mu_i\int_\Omega \mA^{-1}\szi\cdot\szi\,d\mx\leq
H(\mtheta,\sz)\leq H^*.
\]
That upper bound is sharp. Indeed, if $\mthz\in\lTz$ then $(\mthz,\lammz\mI)$ is a maxmizer of $J$ on $\lB$:  by Remark \ref{rem:solT} we have $\szi=\lammz\nabla u^*_i$, where $\uz$ is the solution of \eqref{semt}, so  $J(\mthz,\lammz\mI)=H(\mthz,\sz)$ holds. 

Moreover, for such $\mthz$, a pair $(\mthz,\mA)\in\lB$ is a maximizer of $J$ over $\lB$ if and only if
\begin{equation}\label{eq:optB}
\mA^{-1}\szi=\frac1{\lammz}\szi\,,\;i=1,\ldots,m.
\end{equation}
Indeed, by Remark \ref{rem:solT}, this condition is equivalent to $\mA\nabla u^*_i=\lammz\nabla u^*_i=\szi$, where  $\uz$  is the solution of \eqref{semt} 
corresponding to $\mthz$, but also, by the last equalities, the solution of \eqref{sem}.

Conversely, let $(\mtheta,\mA)\in\lB$ be a maximizer of $J$ over $\lB$. By the same reasoning as above, we obtain
\[
H^*=J(\mtheta,\mA)\leq  \sum_{i=1}^m\mu_i\int_\Omega \mA^{-1}\szi\cdot\szi\,d\mx\leq
H(\mtheta,\sz)\leq H^*.
\]
Hence, all inequalities must be equalities. In particular, $\mtheta$ solves \eqref{eq:opt_des_T}. Consequently, the optimality condition  \eqref{eq:optB} holds  with $\mtheta$ in place of $\mthz$.

\end{remark}

\section{Maximum points of $I$ on $\lT$}

Our goal is now to characterize the set of solutions of problem \eqref{max_1}, under the assumption that the magnitude of the optimal flux field $\sigma^\ast\in\lS$ is known.
In view of  Remark \ref{rem:solT}, this constitutes the first step toward determining solutions of  problem \eqref{eq:opt_des_T}.
To simplify the notation, let us introduce
\[
    \psi:=\sum_{i=1}^m \mu_i |\bm{\sigma}^\ast_i|^2 \in \textup{L}^1(\Omega).
\]
Clearly,  $\psi \geq 0$ almost everywhere on $\Omega$. As mentioned previously, problem  \eqref{max_1} can  be 
equivalently rewritten as
\begin{equation}\label{max_2}
\begin{aligned}
  \int_\Omega \sum_{j=1}^N\frac{\theta_j}{\lambda_1(\mM_j)} \psi\,d\mx&\to\max\\
  \mtheta&\in\lT.
\end{aligned}
\end{equation}
Let us also for a given $f : \Omega \to [0,\infty)$ denote by $\lambda_f : [0,\infty) \to [0,\infty]$ its distribution function, defined as
\[
    \lambda_{f}(\alpha)=\mu(\{ \mx \in \Omega : f(\mx) > \alpha \}).
\]
We emphasize that the distribution function $\lambda_f$ is non-increasing and right-continuous. As such, it is continuous up to countably many points. Moreover, for each $\alpha > 0$ the left-sided limits equate to
\begin{equation}\label{eq:ls_limit}
    \lambda_f(\alpha-) = \mu(\{ \mx \in \Omega : f(\mx) \geq \alpha \}).
\end{equation}
Lastly, the points of discontinuity are precisely those $\alpha>0$ for which the level set $\{ \mx \in \Omega : \psi(\mx)=\alpha\}$ has positive measure.

To this end, under the assumption
\begin{equation}\label{eq:pad}
\lambda_1(\mM_1)<\lambda_1(\mM_2)<\cdots<\lambda_1(\mM_N),
\end{equation}
we introduce a  sequence of nonnegative real numbers $\alpha_1,\ldots,\alpha_N$:
\[
    \alpha_k : = \min \left\{ \alpha \geq 0 : \lambda_{\psi}(\alpha) \leq \sum_{i=1}^k q_i\right\}.
\]
Since $\lambda_\psi$ is right-continuous, these numbers are well-defined (in the sense that the minima above take place) and for each $k=1,\ldots,N$ it holds 
\[
\lambda_\psi(\alpha_k) \leq \sum_{i=1}^k q_i \leq \lambda_\psi(\alpha_k-).
\]
Moreover, since it is non-increasing, we also have $\alpha_{k+1} \leq \alpha_k$ for each $k=1,\ldots,N-1$ with $\alpha_N=0$. 

Necessary conditions for optimality of the solution of optimal design problem \eqref{max_2} are now given in the next result.

\begin{theorem}\label{tm:solT}
    Suppose that \eqref{eq:kol} and \eqref{eq:pad} hold. If $\mthz \in \lT$ is an optimal solution of problem \eqref{max_2}, then, for each $k=1,\ldots,N$ (with the convention $\alpha_0 = +\infty$) the following holds
    \begin{equation}\label{eq:nec_cond}
        \begin{split}
            \supp \thez{k} \subseteq \{ \mx \in \Omega : \alpha_{k} \leq \psi \leq \alpha_{k-1}\}&\\
            \thez{k} = 1 \text{ a.e. on } \{ \mx \in \Omega : \alpha_{k} < \psi < \alpha_{k-1}\}&
        \end{split}
    \end{equation}
    Conversely, any  $\mthz \in \lT$ satisfying conditions \eqref{eq:nec_cond} is an optimal solution of problem \eqref{max_2}.
\end{theorem}

\begin{remark}\label{rem:idea}
The choice of numbers $\alpha_1, \ldots, \alpha_N$ also motivates the formulation and proof of Theorem \ref{tm:solT}. The underlying idea is fairly simple: given \eqref{eq:pad}, a natural approach to solving \eqref{max_2} is to allocate the material corresponding to $\mM_1$ to the part of $\Omega$ where $\psi$ is largest. Once the prescribed amount $q_1$ has been used, one then allocates the material corresponding to $\mM_2$ to the remaining portion of $\Omega$ where $\psi$ is now largest, continuing until the prescribed amount $q_2$ has been reached. At this point, a total area of size $q_1 + q_2$ has been assigned to materials $\mM_1$ and $\mM_2$. This procedure continues analogously until the entire domain $\Omega$ has been covered.

Thus, the numbers $\alpha_1, \ldots, \alpha_N$ are chosen so that cutting at the \textit{horizontal} level sets $\{\psi = \alpha_k\}$ partitions $\Omega$ into the desired regions (see Figures \ref{fig:case1} and \ref{fig:case2}).
\end{remark}

\begin{proof}
    We start by proving the necessity of the condition
    \[
        \theta^\ast_1 = 1 \text{ on } \{ \psi > \alpha_1\}.
    \]
    Assume that $\theta^\ast_1 \neq 1$ on a subset of $\{\psi > \alpha_1\}$ of positive measure. Since $\sum_{i=1}^N \theta^\ast_i=1$ and $\mu(\{\psi > \alpha_1\}) \leq q_1$, we can, for some small $\varepsilon > 0$, find $k \in \{2,\,\ldots,N\}$ and sets $A \subseteq \{ \psi > \alpha_1\}$, $B \subseteq \{ \psi \leq \alpha_1\}$ of equal positive measure such that
    \begin{itemize}
    \item $\thez{1} < 1-\varepsilon$ a.e. on $A$ and $\thez{1} > \varepsilon$ a.e. on $B$
    \item $\thez{k} > \varepsilon$ a.e. on $A$ and $\thez{k} < 1- \varepsilon$ a.e. on $B$.
\end{itemize}
We may then define $\meta \in \Lb{\Omega; \K}$ with
\begin{equation}\label{eq:zamjena}
    \meta = \mthz+\frac{\varepsilon}{2}\left(\chi_A-\chi_B\right)(\me_1-\me_k).
\end{equation}
Note that the amount of each material has remained unchanged, hence $\meta \in \lT$.
It follows
\[
    I(\meta)-I(\mthz)=\left(\frac{1}{\lambda_1(\mM_1)}-\frac{1}{\lambda_1(\mM_k)}\right)\frac{\varepsilon}{2}\int_\Omega (\chi_A-\chi_B)\psi.
\]
Since $\lambda_1(\mM_1) < \lambda_1(\mM_k)$, $\chi_A \psi > \alpha_1\chi_A$ and $\chi_B\psi \leq \alpha_1\chi_B$, we conclude
\[
    I(\meta)-I(\mthz) > 0.
\]
Hence, it is necessary for $\thez{1}$ to be equal to 1 on the set $\{\mx \in \Omega : \psi(\mx) > \alpha_1\}$.

Now, if it holds $\lambda_\psi(\alpha_1)=q_1$, we can conclude that the prescribed condition on the amount of the first material is already satisfied on the set $\{\psi > \alpha_1\}$, hence $\supp \thez{1} \subseteq \{\psi \geq \alpha_1\}$. Otherwise, we have \[\lambda_\psi(\alpha_1-) \geq q_1 > \lambda_\psi(\alpha_1),\]
which further yields
\begin{equation}\label{eq:alpha1}
    \begin{split}
        \mu(\{\psi \geq \alpha_1\}) &=\lambda_\psi(\alpha_1-) \geq q_1\\
        \mu(\{ \psi = \alpha_1 \}) &= \lambda_\psi(\alpha_1-)- \lambda_\psi(\alpha_1) > 0.
    \end{split}
\end{equation} To prove the necessary condition on the support of $\thez{1}$, it remains to show that in this case we have $\thez{1}=0$ on $\{\psi < \alpha_1\}$. 

Assume $\thez{1} \neq 0$ on a subset of $\{ \psi < \alpha_1\}$ of positive measure, and subsequently (because of \eqref{eq:alpha1}) that $\thez{1} \neq 1$ on $\{\psi \geq \alpha_1\}$. This allows us to find subsets $A \subseteq \{\psi = \alpha_1\}$ and $B \subseteq \{\psi < \alpha_1\}$ of equal positive measures such that for some $k \in \{2,\ldots,N\}$ it holds
\begin{itemize}
    \item $\thez{k} > \varepsilon$ a.e. on $A$ and $\thez{k} < 1-\varepsilon$ a.e. on $B$
    \item $\thez{1} < 1- \varepsilon$ a.e. on $A$ and $\thez{1} > \varepsilon$ a.e. on $B$.
\end{itemize}
We may now employ similar arguments to those in the previous case, providing an element $\meta \in \lT$ such that $I(\meta) > I(\mthz)$, which proves the necessity of conditions \eqref{eq:nec_cond} for $k=1$.

Next we turn to the case $k=2$ and its necessary conditions. The proof of this case essentially provides the inductive step, so we present it explicitly for clarity. First, note that the previously proven $\eqref{eq:nec_cond}_2$ for $k=1$ already implies $\supp \thez{2} \subseteq \{\psi \leq \alpha_1\}$. We differentiate two cases: $\alpha_2 < \alpha_1$ and $\alpha_2=\alpha_1$. Let us first assume $\alpha_2 < \alpha_1$. Remainder of the proof of this case is split into three parts:
\begin{enumerate}
    \item[(1)] proving $\thez{1}+\thez{2}=1$ on $\{\psi = \alpha_1\}$
    \item[(2)] proving $\thez{2}=1$ on $\{ \alpha_2 < \psi < \alpha_1\}$
    \item[(3)] proving $\thez{2}=0$ on $\{\psi < \alpha_2\}$.
\end{enumerate}
If it were $\thez{1}+\thez{2} \neq 1$ on $\{\psi = \alpha_1\}$ then there would be a $k \in \{3,\ldots,N\}$ such that for some $\varepsilon > 0$ we have $\thez{k} > \varepsilon$ on some non-negligible subset of $\{\psi = \alpha_1\}$. Since
\begin{equation*}
    \begin{split}
        q_1 + \int_{\{\psi = \alpha_1\}} \thez{2} = \int_{\{\psi \geq \alpha_1\}} \thez{1}+\thez{2} < \mu(\{\psi \geq \alpha_1\}) = \lambda_\psi(\alpha_1-) \leq \lambda_\psi(\alpha_2)\leq q_1+q_2,
    \end{split}
\end{equation*}
we deduce $\int_{\{\psi < \alpha_1\}}\thez{2} > 0$. Therefore, we can find a non-negligible subset of $\{\psi < \alpha_1\}$ of positive measure such that $\thez{2} > \varepsilon$ on it. Once again, due to \eqref{eq:pad}, swapping out $\thez{2}$ and $\thez{k}$ as in \eqref{eq:zamjena} shows that such $\mthz$ is not optimal. This proves (1). Let us now prove (2). Assuming $\thez{2} \neq 1$ on $\{\alpha_2 < \psi < \alpha_1\}$, we also necessarily have $\int_{\{\psi > \alpha_2\}} \thez{2} < q_2$. Therefore, we may again find $k \in \{3,\ldots, N\}$ and non-negligible subsets of $\{\alpha_2 < \psi < \alpha_1\}$ and $\{\psi \leq \alpha_2\}$ on which we can swap some amount of materials corresponding to $\thez{2}$ and $\thez{k}$ to deduce such $\mthz$ is not optimal. Claim (3) follows analogously, since $\thez{2} \neq 0$ on $\{\psi < \alpha_2\}$ implies $\int_{\{\psi \geq \alpha_2\}} \thez{1}+\thez{2}<\mu(\{\psi \geq \alpha_2\})$, which implies another material $\thez{k}$ must appear on $\{\psi \geq \alpha_2\}$, and the same swapping argument ensues. 

This proves necessary conditions for $k=2$ in case of $\alpha_2 < \alpha_1$. In case of $\alpha_2 = \alpha_1$, necessary conditions \eqref{eq:nec_cond} boil down to $\supp \thez{2} \subseteq \{\psi=\alpha_1\}$, so it's enough to show that $\thez{2}=0$ on $\{\psi < \alpha_1\}$. But $\thez{2} \neq 0$ on $\{\psi < \alpha_1\}$ implies $\int_{\{\psi = \alpha_1\}} \thez{1}+\thez{2}< \mu(\{\psi=\alpha_1\})$, so the same swapping argument applies again, which concludes the proof of necessary conditions \eqref{eq:nec_cond} for $k=2$. It is also important to emphasize that the proof of the first two cases also shows that $\thez{1} + \thez{2} = 1$ on ${\psi > \alpha_2}$; this conclusion is essential for establishing the inductive step (see Figures \ref{fig:case1} and \ref{fig:case2}).

As mentioned earlier, this provides the blueprint for the inductive step, so we only sketch it here. Assume that for some $k<N$ we have established the necessary conditions \eqref{eq:nec_cond} for $l=1,\ldots,k$, together with $\sum_{i=1}^l \thez{i}=1$ on $\{\psi > \alpha_l\}$. It immediately follows that $\supp \thez{k+1} \subseteq \{\psi \leq \alpha_{k}\}$.

If $\alpha_{k+1}<\alpha_k$, we mimic steps (1)-(3) from the proof of the case $k=2$, the only modification being that step (1) now reads
\[
    \sum_{i=1}^{k+1} \thez{i}=1 \text{ on } \{\psi = \alpha_k\}.
\]
From \eqref{eq:nec_cond} we obtain
\[
    \sum_{i=1}^k q_i +\int_{\{\psi = \alpha_{k}\}} \thez{k+1} = \int_{\{\psi \geq \alpha_k\}} \sum_{i=1}^{k+1} \thez{i} \leq  \lambda_\psi(\alpha_{k}-) \leq \lambda(\alpha_{k+1}) \leq\sum_{i=1}^{k+1} q_i,
\]
and hence $\int_{\{\psi = \alpha_k\}} \thez{k+1} \leq q_{k+1}$. If there exists $l\in{k+2,\ldots,N}$ such that $\thez{l}\neq 0$ on $\{\psi=\alpha_k\}$, then the final inequality is strict, and a standard swapping argument between materials $k+1$ and $l$ yields a contradiction.

Steps (2) and (3), as well as the case $\alpha_{k+1}=\alpha_k$, are handled with analogous modifications.

To prove that the conditions  \eqref{eq:nec_cond} are also sufficient for optimality in problem \eqref{max_2}, we begin by  noting that they uniquely determine $\mthz$  on open strips $\{\alpha_{k-1} < \psi < \alpha_k\}$. Although they do not ensure  uniqueness on the level sets $\{\psi = \alpha_k\}$,  additional information can be extracted from the supports of $\thez{1},\ldots,\thez{N}$. In particular, on each level set $\{\psi = \alpha_k\}$ it must hold that
    \[
        \sum_{i=k_-}^{k^+} \thez{i}=1,
    \]
where $k_-=\min \{l : \alpha_l=\alpha_k\}$ and $k^+=\max \{l : \alpha_l=\alpha_k\}+1$ (with the convention that $k^+=N$ if this expression equals $N+1$). Hence, only a restricted subset of materials may appear  on such  level sets. 
For any  $k$ and $k_-,k^+$ defined as above, 
rearranging a solution $\mtheta$ of \eqref{max_2} on the set $\{\psi = \alpha_k\}$ yields another solution $\mtht$ of $\eqref{max_2}$, provided that
    \begin{itemize}
        \item $\sum_{i=k_-}^{k^+} \tht_i=1$
        \item $\int_{\{\psi=\alpha_k\}} 
        \tht_i= \int_{\{\psi=\alpha_k\}} \theta_i$ for each $i=k_-,\ldots,k^+$,
    \end{itemize}
where the second condition follows from the volume constraints imposed on each material in the definition of $\lT$.
Indeed, under these assumptions it is straightforward to verify that  $I(\mtheta)=I(\mtht)$. 
This therefore provides a characterization of all solutions of the problem \eqref{max_2}. 
\end{proof}

\begin{figure}[H]
    \centering
    \includegraphics[scale=1]{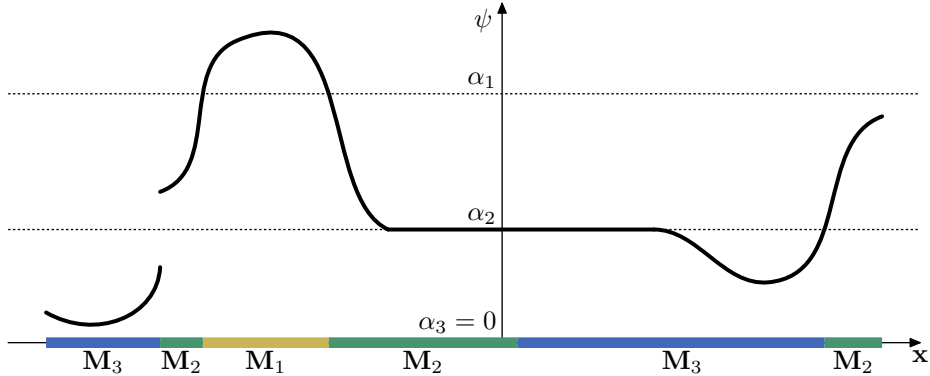}
    \caption{Example of a case where $\alpha_1 > \alpha_2 > \alpha_3$. This is a bang–bang solution in which the open strips are filled with the corresponding materials, while the level set $\psi = \alpha_2$ is non-negligible and is occupied by materials $\mM_2$ and $\mM_3$.}
    \label{fig:case1}
\end{figure}

\begin{figure}[H]
    \centering
    \includegraphics[scale=1]{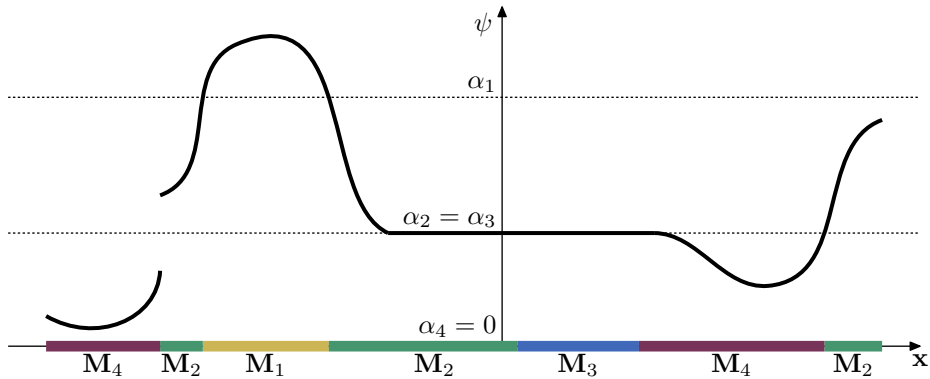}
    \caption{Example where $\alpha_1 > \alpha_2 = \alpha_3 > \alpha_4$. This is a bang–bang solution in which, on the horizontal level set $\psi = \alpha_2 = \alpha_3$, one uses the remaining amount of material $\mM_2$, the fully prescribed material $\mM_3$ (as implied by $\alpha_2=\alpha_3$), and a portion of material $\mM_4$.}
    \label{fig:case2}
\end{figure}


\begin{remark}\label{rem:generalizations}
    As a closing remark, we briefly mention two possible generalizations:
    \begin{enumerate}
        \item[(1)] Relieving the assumption \eqref{eq:pad}. Of course, we may always assume that the corresponding non-strict inequalities hold, by ordering the materials so that their smallest eigenvalues are non-decreasing. In that case, there is no essential distinction in the use of such materials: the terms $\frac{\theta_i}{\lambda_1(\mM_i)}$ coincide, and exchanging any amount of one such material for another does not change the value of $I(\mtheta)$. For instance, if $\lambda_1(\mM_1)=\dots=\lambda_1(\mM_k)$ for some $k$, then, for all practical purposes, any mixture of the materials $\mM_1,\ldots,\mM_k$ may be regarded as a single material with prescribed quantity $\sum_{i=1}^k q_i$, and one may proceed as if \eqref{eq:pad} were satisfied. It is worth noting, however, that in this case the uniqueness of solutions fails even on the open strips determined by $\alpha_1,\ldots,\alpha_N$.
        \item[(2)] 
        Prescribing an upper bound on each material, instead of the exact amount. Namely, for $q_1,\ldots,q_N \geq 0$ such that $\sum_{i=1}^N q_i \geq \mu(\Omega)$, we impose
        \[
            \int_\Omega \theta_i \leq q_i, \quad i=1,\ldots,N.
        \]
        Assuming that the setting of the problem is otherwise unchanged, the approach would follow the same pathway described in Remark~\ref{rem:idea}. Since \eqref{eq:pad} still holds in this setting, one would allocate the materials sequentially and to their full prescribed capacities until the entire domain $\Omega$ has been covered. This may result in some materials being used in smaller-than-prescribed amounts, or even not being used at all. More precisely, for $\widetilde N:= \max\{k : \sum_{i=1}^k q_i < \mu(\Omega)\}$ one prescribes the amount of materials $\mM_1,\ldots,\mM_{\widetilde N}$ to be exactly $q_i$, the amount of material $\mM_{\widetilde N+1}$ to be $\mu(\Omega)-\sum_{i=1}^{\widetilde N} q_i$ and, discarding materials $\mM_{\widetilde N +2},\ldots, \mM_N$, proceeds as in \ref{tm:solT}.
    \end{enumerate}
\end{remark}



\section{Radially Symmetric Problems}

In the sequel, we shall restrict our attention to sets $\Omega$ with spherical symmetry. 
In spherical coordinates, $\Omega$ is represented by points whose radial component $r$ lies within a given interval $\omega$. 
In other words, since we consider only connected sets, $\Omega$ may be either a ball or an annulus. 
Furthermore, we assume that the right-hand sides $f_i$ of the state equations \eqref{sem} are radial functions.

The established uniqueness of the optimal flux $\msigma^*$ in \eqref{minimax} implies that, for each $i=1,\ldots,m$, the corresponding flux $\szi$ is spherically symmetric:
\begin{equation}\label{eq:radsig}
\szi = \sigma_i^*(r) \, \me_r,
\end{equation}
where $\me_r$ denotes the unit vector in the radial direction (pointing outward). Indeed, if $\szi$ were not radial, then for a given saddle point $(\mthz,\sz)$ of $H$, one could construct another saddle point 
\[
(\mtht(\mx),\st(\mx))=(\mthz(\mQ\mx),\mQ^\top \sz(\mQ\mx))
\]
for any orthogonal matrix $\mQ\in\MR d$, which would contradict the uniqueness of $\msigma^*$.
Indeed, $\mtht$ belongs to $\lT$, $\st$ belongs to $\lS$ and
\[
H(\mtht,\st)=\sum_{i=1}^m\mu_i\int_\Omega  \frac{|\mQ^\top\szi(\mQ\mx)|^2}{\lambda_-(\mthz(\mQ\mx))} \,d\mx=H(\mthz,\sz),
\]
where we applied a change of variables in the integral.

As the first step, we look for the set $\lS$, in particular by the conclusion above, we are interested only in spherically symmetric $\msigma$. More precisely, for each $f_i$, we should  determine all
radial functions $\sigma_i$ such that
\begin{equation}\label{eq:divr}
-\frac1{r^{d-1}}\left({r^{d-1}}\sigma_i\right)'=f_i,,\;r\in\omega.
\end{equation}

The case of ball $\omega=[0,R\rangle$ is particularly simple, because if  $r\mapsto r^{\frac{d-1}{2}}f_i(r)$ belong to $\Ld{\oi0R}$, $i=1,\dots,m$, 
then $\lS$ contains a single radial flux, which must therefore coincide with  $\sz$.
Indeed, for every $i$, each equation \eqref{eq:divr} has unique solution in $\Ld{\oi0R}$:
\begin{equation}\label{eq:solsig}
\sigma_i^*(r) = -\frac{1}{r^{d-1}}\int_0^r\rho^{d-1}f_i(\rho)\,d\rho.
\end{equation}
 Consequently, the function $\psi$ appearing in Theorem \ref{tm:solT}
is expressed by formula, allowing an explicit characterization of all  maximizers of \eqref{max_2}.  For any such {\it radial} maximizer 
$\mthz$, the corresponding  state is also radial, and is uniquely determined by
\[
u^*_i(r)=-\int_r^R \frac1\lammzr \sigma^*_i(\rho)\,d\rho.
\]
In view or Remark \ref{rem:solT}, this function  satisfies condition \eqref{eq:optT}, since $\nabla u^*_i=(u^*_i)'\me_r=\frac1\lammz\msigma_i$, implying that $\mthz$ 
solves \eqref{eq:opt_des_T}. Moreover, as each $\sigma^*_i$ is radial, the simple laminate with local proportion $\mthz$ and layers orthogonal to the radial direction has conductivity $\mA$ that satisfies \eqref{eq:optB} and therefore solve both \eqref{odpb} and \eqref{odp}.  


In case of an annulus domain, an integration constant appears in formula \eqref{eq:solsig}, which complicates the direct application of the method, although its use remains feasible \cite{KVjde}.

\section{Examples}

\subsection{Single State Problem with Unique and Classical Solution}\label{pr1}

To begin, we consider a simple example of an energy maximization problem involving a single state equation.
Let $\Omega:=B(\mnul,1)\subseteq \R^2$ be the unit ball, which in polar coordinates is described by $r\in\omega:=[0,1\rangle$.
We study an optimal design problem in which the source term
$f$ is a piecewise constant function given by
\[
f={\chi}_{B(0,\frac12)}.
\]
We consider the problem of arranging three anisotropic materials within $\Omega$, to maximize the energy functional, subject to prescribed volume fractions $\eta_j\in\oi 01$, summing to one, so that $q_j=\eta_j\mu(\Omega)$ for each material.

As noted earlier, the solution of the optimization problem 
does not depend on the specific values of the minimal eigenvalues of $\mM_1,\mM_2$ and $\mM_3$ provided, for instance, that  $\lambda_1(\mM_1)<\lambda_1(\mM_2)<\lambda_1(\mM_3)$. If some of these eigenvalues coincide, the conclusions can be adapted in the sense of Remark \ref{rem:generalizations} (1). Likewise, if the material constraints are formulated  as inequalities rather than equalities, the discussion of Remark \ref{rem:generalizations} (2) applies.

The unique radial flux in $\lS$ is
\[
\sz(r)=\left\{
{\renewcommand{\arraystretch}{2.}
\begin{array}{rl}
-\displaystyle\frac r2,&0\leq r\leq \frac12\\
\displaystyle-\frac1{8r},&\frac12\leq r\leq 1,
\end{array}
}
\right.
\]
leading to $\psi=(\sz)^2$, whose graph is presented
in Figure \ref{figa}. Since the graph contains no horizontal segments, Theorem \ref{tm:solT} yields a unique solution (up to sets of measure zero) of problem \eqref{max_2}. 
This solution is radial and a bang-bang solution, obtaining only values 0 or 1 for each component of $\mthz$. By Remark \ref{rem:J2} we conclude that $(\mthz,\lammz\mI)$ is a solution of problem \eqref{odpb}, but since it is a classical solution (using only original pure materials), also a solution of \eqref{odp}. Moreover,
by the same Remark, it follows that it is the unique solution of \eqref{odpb}, and consequently \eqref{odp}.

\begin{figure}
\centering
\includegraphics[scale=1]{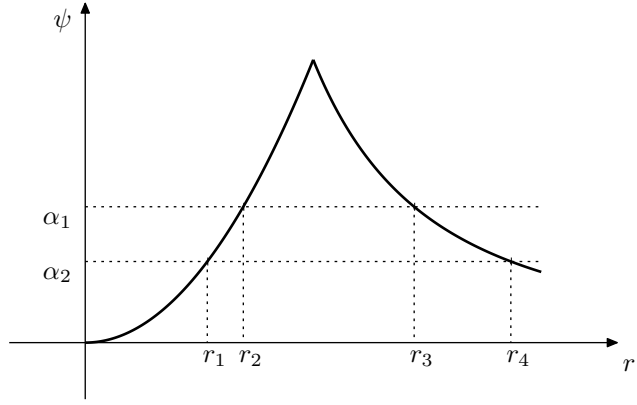}
\caption{Graph of the function $\psi$ for the example in Subsection \ref{pr1}, with the radii indicating the interfaces between the materials}\label{figa}
\end{figure}

For simplicity, we assume that $\alpha_2\geq\psi(1)$ (see Figure \ref{figa}) which is equivalent to the condition $\eta_1+\eta_2\geq\frac{15}{16}$.

Let us describe the optimal material arrangement. The first material occupies the annular region $r\in[r_2,r_3]$, where the radii $r_2$ and $r_3$ are determined by the following system of equations
\[
r_3^2-r_2^2=\eta_1\,\quad \frac{r_2^2}4=\frac1{64r_3^2}.
\]

The second material occupies the regions $r\in[r_1,r_2]\cup[r_3,r_4]$, while the third material is placed in the regions $r\in[0,r_1]\cup[r_4,1]$. The radii $r_1$ and $r_4$ are determined by the equations
\[
r_4^2-r_1^2=\eta_1+\eta_2\,,\quad\frac{r_1^2}4=\frac1{64r_4^2}.
\]

For instance, if $\eta_1=\eta_2=0.4$ (and hence $\eta_3=0.2$), these equations yield the following values for radii:
\[
{\renewcommand{\arraystretch}{2.3}
\begin{array}{ll}
\displaystyle r_1=\sqrt{\frac{-0.8+\sqrt{0.89}}2}\approx 0.2678,\qquad
&\displaystyle r_2=\sqrt{\frac{-0.4+\sqrt{0.41}}2}\approx 0.3466\\
\displaystyle r_3=\sqrt{\frac{0.4+\sqrt{0.41}}2}\approx 0.7212,\qquad
&\displaystyle r_4=\sqrt{\frac{0.8+\sqrt{0.89}}2}\approx 0.9336\,.
\end{array}
}
\]

\subsection{Appearance of Nonunique Classical Designs}\label{pr2}

Let $\Omega=B(\mnul,2)$ and consider a source term
\[
f(r)=\left\{
{\renewcommand{\arraystretch}{1.2}
\begin{array}{rl}
-2,&0\leq r\leq 1\\
\displaystyle-\frac1r,&1< r\leq 2.
\end{array}
}
\right.
\]
The function $\psi$ appearing in optimization problem \eqref{max_2} is given by
\[
\psi(r)=\left\{
{\renewcommand{\arraystretch}{1.2}
\begin{array}{rl}
r^2,&0\leq r\leq 1\\
1,&1< r\leq 2.
\end{array}
}
\right.
\]

We seek  an arrangement of three (possibly anisotropic) materials with the smallest eigenvalues satisfying $\lambda_1(\mM_1)<\lambda_1(\mM_2)<\lambda_1(\mM_3)$, and prescribed volume fractions
$q_j=\eta_j\mu(\Omega)$, $j=1,\ldots, N$, where $\eta_j\in\oi01$ and $\sum_{j=1}^N\eta_j=1$.

Owing to the simple structure of $\psi$, three distinct cases may occur:
\begin{enumerate}
    \item $\displaystyle\eta_1>\frac34$ (equivalently, $\alpha_1<1$)
    \item $\displaystyle\eta_1\leq\frac34$ and $\displaystyle\eta_1+\eta_2>\frac34$ (equivalently, $1=\alpha_1>\alpha_2$)
    \item $\displaystyle\eta_1+\eta_2\leq\frac34$ (equivalently, $\alpha_1=\alpha_2=1$).
\end{enumerate}

The analysis of the first case is parallel to that of the previous example with the same conclusion: the solution of the relaxed problem \eqref{odp} is unique, classical and radial.

In contrast, in the second and third cases the solutions of problem \eqref{max_2} are not unique. Moreover, by the discussion in the previous section, any such radial maximizer is also a solution of the optimization problems \eqref{odp}, \eqref{odpb}, and \eqref{eq:opt_des_T}.

By Theorem~\ref{tm:solT}, all such solutions can be characterized explicitly. In what follows, we restrict attention to radial solutions, as non-radial solutions fail to satisfy condition \eqref{eq:optT}, and hence cannot be solutions of \eqref{odp} (this could occur only if $\sigma^*=0$ on a set of positive measure, in which case the same argument  as in \cite[Subsection 4.2]{Vsiam} applies).

In the second case, the optimal configurations place the third material in an inner ball $B(\mnul,r_1)$ of volume $q_3$, which is surrounded by an annulus with inner radius $r_1$ and outer radius 1 filled entirely with the second material. The outer annulus with radii between  
1 and 2 is occupied by  the first two materials arbitrarily (i.e. $\theta_3=0$ and $\theta_1+\theta_2=1$ in this region), subject to the prescribed volume constraints.

In  the third case, the inner ball $B(0,1)$ is completely occupied by the third material, while in the remaining part of the domain $\Omega$ arbitrary mixtures of all three materials may occur, provided that the volume constraints for each  material are satisfied.

\subsection{Multiple State Problem}\label{pr3}

Let us now consider the case  $m=2$, with $\Omega=B(\mnul,1)\subseteq\R^2$, and source terms
\begin{equation}\label{ff}
f_1:=\chi_{B(\mnul,\frac12)}\;\hbox{ and }\; f_2:=\chi_{B(\mnul,\frac12)^c}\,,
\end{equation}
with weights $\mu_1=3$ and $\mu_2=1$. We consider three materials, with $\lambda_1(\mM_1)<\lambda_1(\mM_2)<\lambda_1(\mM_3)$ and prescribed volume fractions
$q_j=\eta_j\mu(\Omega)$, $j=1,\ldots, N$, where $\eta_j\in\oi01$ and $\sum_{j=1}^N\eta_j=1$, as in the previous examples.


It is straightforward to compute the function  
$
\psi=\mu_1 {\sigma_1^\ast}^2+\mu_2 {\sigma_2^\ast}^2$,
where $\sigma_1^\ast, \sigma_2^\ast\in\pLd{\Omega;\R^2}$ are the unique solutions of 
$-\dv\sigma_i=f_i$, $i=1,2$:
\[
\psi(r)=\left\{
{\renewcommand{\arraystretch}{2.1}
\begin{array}{rl}
\displaystyle\frac{3 r^2}{4},&0\leq r\leq \frac12\\
\displaystyle\frac{4r^4-2r^2+1}{16 r^2},&\frac12\leq r\leq 1.
\end{array}
}
\right.
\]
The graph of $\psi$ is shown in Figure \ref{figb}.

\begin{figure}
\centering
\includegraphics[scale=1]{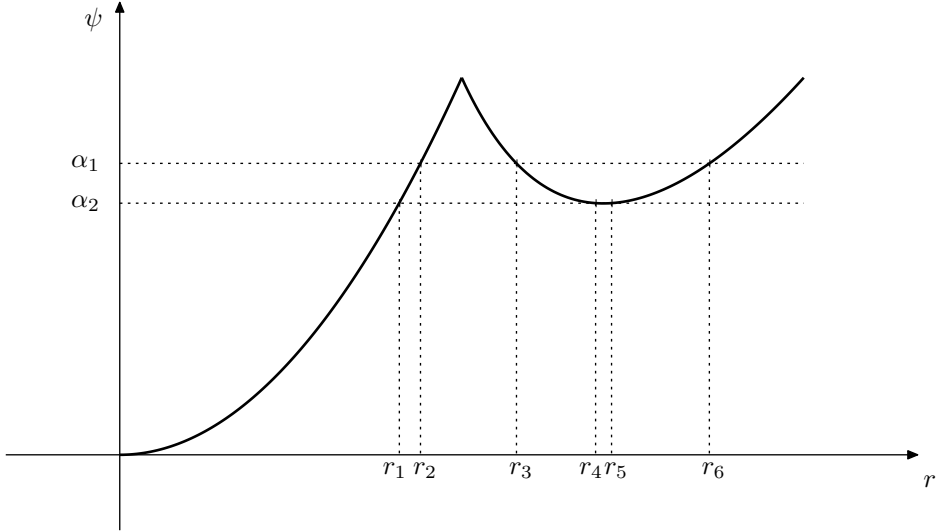}
\caption{Graph of the function $\psi$ for the example in Subsection \ref{pr3}, with the radii indicating the interfaces between the materials}\label{figb}
\end{figure}

Since $\psi$ is piecewise strictly monotone, the same analysis  as in Section \ref{pr1} applies.
In particular, the relaxed problem \eqref{odp} admits a unique solution, which is classical and radially symmetric.

For simplicity, assume that $\eta_1+\eta_2<\frac56$, or equivalently $\alpha_2>\frac18$, where $\frac18$ is the minimum value of $\psi$ on $\left[\frac12,1\right]$ (see Figure \ref{figb}).
The first material then occupies the regions $r\in[r_2,r_3]\cup[r_6,1]$, where the radii $r_2$, $r_3$, and $r_6$ are determined by a  system of algebraic equations (with $r_3<r_6$)
\[
\begin{aligned}
\displaystyle r_3^2-r_2^2+1-r_6^2&=\eta_1\\
\displaystyle\frac{3r_2^2}4=\frac{4r_3^4-2r_3^2+1}{16 r_3^2}&=\frac{4r_6^4-2r_6^2+1}{16 r_6^2}.
\end{aligned}
\]

The second material is distributed over $r\in[r_1,r_2]\cup[r_3,r_4]\cup[r_5,r_6]$, while the third material occupies $r\in[0,r_1]\cup[r_4,r_5]$. Here, $r_1$, $r_4$, and $r_5$ satisfy (with $r_4<r_5$)
\[
\begin{aligned}
\displaystyle r_4^2-r_1^2+1-r_5^2&=\eta_1+\eta_2\\
\displaystyle\frac{3r_1^2}4=\frac{(r_4^2-\frac14)^2+\frac3{16}}{4r_4^2}&=\frac{(r_5^2-\frac14)^2+\frac3{16}}{4r_5^2}.
\end{aligned}
\]
Although the explicit computation is somewhat tedious, it involves solving only quadratic equations and yields a unique solution. In the particular case $\eta_1=\eta_2=0.4$, the resulting radii are listed in Table \ref{tab1}.

\begin{table}
\caption{Approximate values of the  optimal radii for the example in Subsection~\ref{pr3} with  $\eta_1=\eta_2=0.4$.}\label{tab1}
\begin{center}
\begin{tabular}{ c  c  c  c  c  c  c  }
$r_1$ &$r_2$ & $r_3$ & $r_4 $ & $r_5$ & $r_6$ \\
\hline
0.4085&0.4395& 0.5800&0.6955&0.7189&0.8621\\
\end{tabular}
\end{center}
\end{table}

\section*{Acknowledgements}
The research  has been supported in part by Croatian Science Foundation under the project IP-2022-10-7261. 
Ivan Ivec acknowledges support from the project “Implementation of cutting-edge research and its application as part of the Scientific Center of Excellence for Quantum and Complex Systems, and Representations of Lie Algebras”, Grant No. PK.1.1.10.0004, co-financed by the European Union through the European Regional Development Fund - Competitiveness and Cohesion Programme 2021-2027.
Matko Grbac has been supported in part by Croatian Science Foundation under the project   UIP-2025-02-1337.

\printbibliography




\end{document}